\documentclass[draft]{amsart}
\usepackage{amssymb,graphicx}
\usepackage{multirow,comment,caption}

\newcommand{\df}{\dfrac}
\newcommand{\tf}{\tfrac}

\renewcommand{\i}{\infty}

\newcommand{\beqs}{\begin{equation*}}
\newcommand{\eeqs}{\end{equation*}}
\numberwithin{equation}{section}
 \theoremstyle{plain}
\newtheorem{theorem}{Theorem}[section]
\newtheorem{example}[theorem]{Example}

\newtheorem{corollary}[theorem]{Corollary}
\newtheorem{definition}[theorem]{Definition}

\theoremstyle{remark}

\begin{document}

\makeatletter
\def\imod#1{\allowbreak\mkern10mu({\operator@font mod}\,\,#1)}
\makeatother

\author{Frank Patane}
   \address{Department of Mathematics, University of Florida, 358 Little Hall, Gainesville FL 32611, USA}
   \email{frankpatane@ufl.edu}

\title[\scalebox{.82}{An identity connecting theta series associated with binary quadratic forms of discriminant $\Delta$ and $\Delta p^2$}]{An identity connecting theta series associated with binary quadratic forms of discriminant $\Delta$ and $\Delta($\MakeLowercase{prime}$)^2$}
     
\begin{abstract} 
We state and prove an identity which connects theta series associated with binary quadratic forms of idoneal discriminants $\Delta$ and $\Delta p^2$, for $p$ a prime. Employing this identity, we extend the results of Toh \cite{toh} by writing the theta series of forms of discriminant $\Delta p^2$ as a linear combination of Lambert series. We then use these Lambert series decompositions to give explicit representation formulas for the forms of discriminant $\Delta p^2$. Lastly, we give a generalization of our main identity, which employs a map of Buell \cite{buell} to connect forms of discriminant $\Delta$ to $\Delta p^2$. Our generalized identity links theta series associated with a single form of discriminant $\Delta$ to a theta series associated with forms of discriminant $\Delta p^2$, where $\Delta$ and $\Delta p^2$ are no longer required to be idoneal.
\end{abstract}



\date{\today}
   
\maketitle

   \section{Introduction}
\label{intro}

We use the standard notation of $(a,b,c)$ to represent the class of binary quadratic forms which are equivalent to the binary quadratic form $ax^2 +bxy +cy^2$. Equivalence of two forms means the transformation matrix which connects them is in $SL(2,\mathbb{Z})$. The discriminant of $(a,b,c)$ is defined as $\Delta:= b^2-4ac$, and we only consider the case $\Delta <0$ and $a>0$. The set of all classes of primitive forms of discriminant $\Delta$ comprise what is known as the class group of discriminant $\Delta$, denoted CL$(\Delta)$. It is well known that CL$(\Delta)$ is a finite abelian group under Gaussian composition of forms \cite{buell}.\\

We represent the theta series associated to $(a,b,c)$ as
\[
(a,b,c,q):=\sum_{x,y}q^{ax^2+bxy+cy^2}=\sum_{n\geq0}(a,b,c;n)q^n,
\]
where we use $(a,b,c;n)$ to denote the total number of representations of $n$ by $(a,b,c)$. We define the projection operator $P_{m,r}$ to be 
\[
P_{m,r}\sum_{n \geq 0}a(n)q^n = \sum_{n\geq 0}a(mn+r)q^{mn+r},
\]
where we take $0\leq r<m$. If we ever use the operator $P_{m,r}$ and $r \geq m$, then one can replace $r$ with $k$ where $k \equiv r \imod{m}$ and $0 \leq k<m$. 
Informally, the operator $P_{m,r}$ applied to $(a,b,c,q)$ collects the terms of $(a,b,c,q)$ which have the exponent of $q$ congruent to $r\imod{m}$. We also define
\[
[q^k]\sum_{n\geq 0}a(n)q^n=a(k).
\]
We use the standard notations along with well known identities to establish
\begin{equation}
   \label{poch}
   (a;q)_\i:= \prod_{n=0}^{\infty}(1-aq^n),
   \end{equation}
   \begin{equation}
   \label{E}
   E(q):= (q;q)_\i,
   \end{equation}
   \begin{equation}
   \label{phi}
   \phi(q):= \sum_{n=-\infty}^{\infty}q^{n^2} = \df{E^5(q^2)}{E^2(q^4)E^2(q)},
   \end{equation}
   \begin{equation}
   \label{psi}
   \psi(q):= \sum_{n=-\infty}^{\infty}q^{2n^2-n} = \df{E^2(q^2)}{E(q)},
   \end{equation}
	and
	\begin{equation}
   \label{em}
   E(-q) = \df{E^3(q^2)}{E(q^4)E(q)}.
   \end{equation}
	
We now give a brief discussion of the genus theory of binary quadratic forms. This theory was first developed in section 5 of the famous Disquisitiones Arithmeticae of Gauss, published in 1801 \cite{gauss}. We also note that both \cite{buell} and \cite{cox} are excellent resources which discuss the genus theory as well as other topics which we refer to later.\\

Two binary quadratic forms of discriminant $\Delta$ are said to be in the same genus if they are equivalent over $\mathbb{Q}$ via a transformation matrix in $SL(2,\mathbb{Q})$ whose entries have denominators coprime to $2\Delta$. An equivalent definition for the genera of binary quadratic forms is given by introducing the concept of assigned characters. The assigned characters of a discriminant $\Delta$ are the functions $\left(\tf{r}{p}\right)$ for all odd primes $p\mid \Delta$, as well as possibly the functions $\left(\tf{-1}{r}\right)$, $\left(\tf{2}{r}\right)$, and $\left(\tf{-2}{r}\right)$. The exact details are given in Buell \cite{buell} as well as in Cox \cite{cox}.\\
The genera are of equal size and partition the class group. The number of genera is always a power of 2 and is easily computed given the prime factorization of the discriminant. The number of genera of discriminant $\Delta p^2$ is either equal to the number of genera of discriminant $\Delta$ or double the number of genera of discriminant $\Delta$. Here, and everywhere, $p$ is a prime. Letting $v(\Delta)$ be the number of genera of discriminant $\Delta$ we have
\begin{equation}
\label{numgenform}
\frac{v(\Delta p^2)}{v(\Delta)} = \left\{ \begin{array}{ll}
        1&  2<p, p\mid\Delta,\\
				2&  2<p, p\nmid\Delta,\\
				1&  p=2, p\nmid\Delta,\\
				1&  p=2, \Delta=-4,\\
				2&  p=2, -4\neq\Delta \equiv 4,8 \imod{16},\\
				1&  p=2, \Delta \equiv 12 \imod{16},\\
				1&  p=2, \Delta \equiv 0 \imod{32},\\
				2&  p=2, \Delta \equiv 16 \imod{32}.\\
     \end{array}
     \right.
\end{equation}
We use $h(\Delta):=|$CL$(\Delta)|$ to be the class number of $\Delta$. There is a simple relation between the class number of $\Delta$ and $\Delta p^2$ which is given by 
\begin{equation}
\label{hrel}
h(\Delta p^2)=\df{h(\Delta)\left(p-\left(\tf{\Delta}{p}\right)\right)}{w},
\end{equation}
where
\[
	w :=\left\{ \begin{array}{ll}
        3&  \Delta  =-3,\\
				   2&  \Delta =-4,\\
					   1&  \Delta  <-4.\\
     \end{array}
     \right.
	\]
When each genus contains exactly one class we say the discriminant is idoneal. It is not hard to see that \eqref{hrel} implies $h(\Delta)\mid h(\Delta p^2)$. Formula \eqref{numgenform} implies $v(\Delta)\mid v(\Delta p^2)$. Thus the number of forms in a genus of $\Delta$ divides the number of forms in a genus of $\Delta p^2$. Hence $\Delta p^2$ idoneal implies $\Delta$ must also be idoneal.

\begin{definition}
\label{corr}
Let $g$ be a genus of discriminant $\Delta$ and $G$ a genus of discriminant $\Delta p^2$. Also let $r_1, r_2$ be coprime to $\Delta p^2$ and represented by the genera $g,G$, respectively. We define the map $\Phi_p$, by $\Phi_p(G)= g$ if $\chi(r_1)=\chi(r_2)$ for all assigned characters $\chi$ of $\Delta$. If $\Phi_p(G)=g$ we say $G$ corresponds to $g$.
\end{definition}
Before we give an explicit example to illustrate Definition \ref{corr}, we consider the case $p$ coprime to $2\Delta$. By \eqref{numgenform} we know there are twice as many genera of discriminant $\Delta p^2$ as there are of discriminant $\Delta$. Thus $\Phi_p$ would be a two to one map in this case. If $\chi_1, \ldots \chi_k$ are the assigned characters of discriminant $\Delta$, then the assigned characters of discriminant $\Delta p^2$ are $\chi_1, \ldots \chi_k, \left(\tf{\bullet}{p}\right)$. Let $g$ be a genus of discriminant $\Delta$ and have assigned character vector $\left\langle \chi_1, \ldots \chi_k\right\rangle$. Then the two genera $G_1, G_2$ of discriminant $\Delta p^2$ which correspond to $g$ are the genera with character vectors $\left\langle \chi_1, \ldots \chi_k, +1\right\rangle$ and $\left\langle \chi_1, \ldots \chi_k,-1\right\rangle$.

We now give an example to illustrate the concept of corresponding genera.

\begin{example}\label{excor}\end{example}
\begin{center}
\begin{tabular}{ | l | l | l | l | }
  \hline     
  \multicolumn{2}{|c|}{CL$(-20)$}& $\left(\tf{r}{5}\right)$ & $\left(\tf{-1}{r}\right)$\\
  \hline                   
   $g_1$ & $(1,0,5)$ &$+1$ & $+1$ \\ \hline 
	$g_2$ & $(2,2,3)$ &$-1$ & $-1$ \\ \hline 
\end{tabular}
\begin{flushright}
			 .
			 \end{flushright}
   \end{center}

\begin{center}
\begin{tabular}{ | l | l | l | l | l | }
  \hline     
  \multicolumn{2}{|c|}{CL$(-20\cdot 9) $} & $\left(\tf{r}{5}\right)$ & $\left(\tf{-1}{r}\right)$ & $\left(\tf{r}{3}\right)$\\
  \hline                   
   $G_1$ & $(1,0,45)$  & $+1$ & $+1$&$+1$ \\ \hline 
	$G_2$ & $(5,0,9)$ & $+1$ & $+1$&$-1$ \\ \hline 
	$G_3$ & $(7,4,7)$  & $-1$ & $-1$ &$+1$\\ \hline 
	$G_4$ & $(2,2,23)$  & $-1$ & $-1$&$-1$ \\ \hline 
\end{tabular}
\begin{flushright}
			 .
			 \end{flushright}
   \end{center}

The above tables give the reduced classes of forms for discriminants $-20$ and $-20 \cdot 9$, along with the genus structure and assigned characters. Using Definition \ref{corr}, we see $G_1$ and $G_2$ correspond to $g_1$ since $\left(\tf{r}{5}\right) = \left(\tf{-1}{r}\right) =+1$ for these genera. Similarly $G_3$ and $G_4$ correspond to $g_2$. We now state a central result of the paper, the proof of which is given in Section \ref{main}.

\begin{theorem}
	\label{idoo}
	 Let $(a,b,c)$ be of discriminant $\Delta$ and alone in its genus, $g$. Let $(A,B,C)$ be of discriminant $\Delta p^2$ and alone in its genus $G$. If $\Phi_p(G)=g$ (see Definition \ref{corr}), then for $p$ odd we have
\begin{equation}
\label{idom}
 w(A,B,C,q)
=  w(a,b,c,q^{p^{2}}) +\sum_{i=1}^{p-1} \tf{\left(\tf{ri}{p}\right) +1}{2}P_{p,i}(a,b,c,q),
\end{equation}
 and for $p=2$,
\[
 w(A,B,C,q)
=  w(a,b,c,q^{4}) +P_{2^{t+1},r}(a,b,c,q),
\]
\noindent
where\\
\[
	w :=\left\{ \begin{array}{ll}
        3&  \Delta  =-3,\\
				   2&  \Delta =-4,\\
					   1&  \Delta  <-4,\\
     \end{array}
     \right.
	\]
$r$ is coprime to $\Delta p^2$ and is represented by $(A,B,C)$. When $\Delta \equiv 0 \imod{16}$ we define $t=2$, and for $\Delta \not\equiv 0 \imod{16}$ we define $t=0,1$ according to whether $\Delta$ is odd or even.
\end{theorem}
We note that $t$ is the order of 4 in $\Delta$ when $64\nmid \Delta$. For equation \eqref{idom}, the coefficient of $P_{p,i}(a,b,c,q)$ is either $0$ or $1$. Moreover, if $p$ is coprime to $2\Delta$, then Theorem \ref{idoo} yields two identities for $(a,b,c)\in CL(\Delta)$, with one identity having exactly the terms $P_{p,k}(a,b,c,q)$, where $\left(\tf{k}{p}\right)=1$, and the other identity having exactly the terms $P_{p,j}(a,b,c,q)$, where $\left(\tf{j}{p}\right)=-1$. A generalized version of Theorem \ref{idoo} is Theorem \ref{newthm} given in Section \ref{main2}.\\

Continuing Example \ref{excor}, we apply Theorem \ref{idoo} to $\Delta=-20$ and $p=3$ to find
\begin{align}
(1,0,45,q)&=(1,0,5,q^{9})+P_{3,1}(1,0,5,q),\label{aa1}\\
(5,0,9,q)&=(1,0,5,q^{9})+P_{3,2}(1,0,5,q),\label{aa2}\\
(7,4,7,q)&=(2,2,3,q^{9})+P_{3,1}(2,2,3,q),\label{aa3}\\
(2,2,23,q)&=(2,2,3,q^{9})+P_{3,2}(2,2,3,q)\label{aa4},
\end{align}
and see that this example illustrates the discussion of the preceding paragraph.\\

 Let $f$ be a binary quadratic form of idoneal fundamental discriminant $\Delta$. In \cite{toh}, Toh gives identities which write the theta series associated to $f$ as a linear combination of Lambert series. Theorem \ref{idoo} connects theta series associated with the nonfundamental discriminant $\Delta p^2$ to theta series associated with discriminant $\Delta$. When $\Delta$ is fundamental, one can use Theorem \ref{idoo} in conjunction with the Lambert series decompositions given in \cite{toh} to write the Lambert series decompositions for the theta series associated with the genera of $\Delta p^2$.\\

Let
\[
L_{1}(\Delta,q) = \sum_{n>0}\left(\frac{\Delta}{n}\right)\frac{q^{n}}{1-q^{n}},
\]
and
\[
L_{2}(a,b,q) = \sum_{n>0}\sum_{m=1}^{|b|-1}\left(\frac{a}{n}\right)\left(\frac{b}{m}\right)\frac{q^{nm}}{1-q^{|b|n}}.
\]
If $\Delta$ is a fundamental idoneal discriminant with class number 2, then $-\Delta=tp$ where $p$ is an odd prime, and $t$ is either an odd prime or a power of 2. If $(a,b,c)$ has discriminant $\Delta$, and $(a,b,c)$ represents $m$ with gcd$(m,2\Delta)$=1, we have
\begin{equation}
\label{cons}
(a,b,c,q) = 1 +L_{1}(\Delta,q)
 +\left(\frac{m}{p}\right) L_{2}\left(\left(\frac{-1}{p}\right)p,-\left(\frac{-1}{p}\right)t,q\right).
\end{equation}
We note that Toh takes one page to write the Lambert series decompositions given by \eqref{cons} \cite[pg. 232]{toh}. Also there is a minor typo for the Lambert series decomposition of discriminant $-427$ given in \cite[pg. 232]{toh}, and so we prefer \eqref{cons}.\\

In Section \ref{berkyes} we employ identities of \cite{berk1} which write certain eta-quotients as Lambert series to derive identities of similar type given in \cite{luo}. In Section \ref{main} we give the proof of Theorem \ref{idoo} for known idoneal discriminants. In Section \ref{example} we demonstrate the utility of Theorem \ref{idoo} in deriving an explicit formula for the number of representations by an idoneal form. Section \ref{main2} contains a generalized version of Theorem \ref{idoo}. We give illustrative examples as well as deduce a corollary which is useful in obtaining a Lambert series decomposition for a genus of a discriminant $\Delta p^2$. We then give an example of this corollary as well as derive a new eta-quotient identity.

\section{Connecting identities of \cite{berk1} and \cite{luo}}
\label{berkyes}

In \cite{luo} Z.G. Liu derives the two eta quotient identities
\begin{equation}
\label{eta1}
1+\sum_{n=1}^\infty \left(\df{-15}{n}\right)\df{q^n}{1-q^{n}} = \df{E^2(q^3)E^{2}(q^{5})}{E(q)E(q^{15})},
\end{equation}
and
\begin{equation}
\label{eta2}
\sum_{n=1}^\infty \left(\df{5}{n}\right)\df{q^n -q^{2n}}{1-q^{3n}} = q\df{E^2(q)E^{2}(q^{15})}{E(q^3)E(q^5)}.
\end{equation}

Remark 3.1 of \cite{luo} notes the similarity between \eqref{eta1}, \eqref{eta2}, and the following eta quotients which are derived in \cite{berk1}
\begin{equation}
\label{beta1}
P(q):=1-\sum_{n=1}^\infty \left(\df{-15}{n}\right)\df{q^n}{1+q^{n}} = \df{E(q)E(q^6)E(q^{10})E(q^{15})}{E(q^2)E(q^{30})},
\end{equation}
and
\begin{equation}
\label{beta2}
Q(q):=\sum_{n=1}^\infty \left(\df{5}{n}\right)\df{q^n +q^{2n}}{1+q^{3n}} = q\df{E(q^2)E(q^3)E(q^{5})E(q^{30})}{E(q^6)E(q^{10})}.
\end{equation}
In this section we derive the connection between \eqref{eta1}, \eqref{beta1}, and \eqref{eta2}, \eqref{beta2}.\\

We derive \eqref{eta1} by letting $q \to -q$ in \eqref{beta1} and comparing the even part of both sides of the resulting equation. We have
\begin{equation}
\label{d1}
P_{2,0}P(-q)=P_{2,0}P(q)=1-P_{2,0}\sum_{n=1}^\infty \left(\df{-15}{n}\right)\df{q^n}{1+q^{n}}.
\end{equation}
Employing \eqref{psi}, \eqref{em}, and \eqref{beta1}, we have
\begin{equation}
\label{d2}
P(-q)=\psi(q)\psi(q^{15})\df{E(q^6)E(q^{10})}{E(q^4)E(q^{60})}.
\end{equation}
One can directly prove the identity
\begin{equation}
\label{d3}
P_{2,0}\psi(q)\psi(q^{15})=\psi(q^6)\psi(q^{10}),
\end{equation}
but instead we appeal to \cite[pg. 377]{berndt}. Combining \eqref{d2} and \eqref{d3} yields
\begin{equation}
\label{d4}
P_{2,0}P(-q)=\psi(q^6)\psi(q^{10})\df{E(q^6)E(q^{10})}{E(q^4)E(q^{60})}=\df{E^2(q^{12})E^2(q^{20})}{E(q^4)E(q^{60})}.
\end{equation}
Also we have
\begin{equation}
\label{m1}
\begin{aligned}
1-P_{2,0}\sum_{n=1}^\infty \left(\df{-15}{n}\right)\df{q^n}{1+q^{n}}&=1-P_{2,0}\sum_{n=1}^\infty \left(\df{-15}{n}\right)\df{q^n-q^{2n}}{1-q^{2n}}\\
&=1-\sum_{n=1}^\infty \left(\df{-15}{n}\right)\df{q^{2n}}{1-q^{4n}}+\sum_{n=1}^\infty \left(\df{-15}{n}\right)\df{q^{2n}}{1-q^{2n}}\\
&=1-\sum_{n=1}^\infty \left(\df{-15}{n}\right)\df{q^{2n}}{1-q^{4n}}+\sum_{n=1}^\infty \left(\df{-15}{n}\right)\df{q^{2n}+q^{4n}}{1-q^{4n}}\\
&=1+\sum_{n=1}^\infty \left(\df{-15}{n}\right)\df{q^{4n}}{1-q^{4n}}.
\end{aligned}
\end{equation}
Combining \eqref{d1}, \eqref{d4}, and \eqref{m1}, we have
\begin{equation}
\label{d5}
1+\sum_{n=1}^\infty \left(\df{-15}{n}\right)\df{q^{4n}}{1-q^{4n}}=\df{E^2(q^{12})E^2(q^{20})}{E(q^4)E(q^{60})}.
\end{equation}
Letting $q^4 \to q$ in \eqref{d5} yields \eqref{eta1}, and we have derived Liu's \cite[Prop. 3.3]{luo} from \cite[Thm 5.1]{berk1}.\\
In a similar way we derive \eqref{eta2} from \eqref{beta2}.
We begin by employing \eqref{psi}, \eqref{em}, and \eqref{beta2}, to find
\begin{equation}
\label{q11}
P_{2,0}Q(q)=P_{2,0}Q(-q)=-P_{2,0}~q\psi(q^3)\psi(q^5)\df{E(q^{2})E(q^{30})}{E(q^{12})E(q^{20})}.
\end{equation}

The equation
\begin{equation}
\label{qd3}
P_{2,0}~q\psi(q^3)\psi(q^{5})=q^4\psi(q^2)\psi(q^{30}),
\end{equation}
can be shown directly, and it also appears in \cite[pg. 377]{berndt}. Combining \eqref{psi}, \eqref{q11} and \eqref{qd3} yields
\begin{equation}
\label{q1}
P_{2,0}Q(q)=-q^4\df{E^2(q^{4})E^2(q^{60})}{E(q^{12})E(q^{20})}.
\end{equation}
Employing \eqref{beta2} and using techniques similar to those of \eqref{m1}, we find
\begin{equation}
\label{q2}
P_{2,0}Q(q)=P_{2,0}\sum_{n=1}^\infty \left(\df{5}{n}\right)\df{q^n +q^{2n}}{1+q^{3n}}=-\sum_{n=1}^\infty \left(\df{5}{n}\right)\df{q^{4n} -q^{8n}}{1-q^{12n}},
\end{equation}
Equating \eqref{q1} and \eqref{q2}, and letting $q^4 \to q$ yields \eqref{eta2} as desired.\\

 We note that the Lambert series of \eqref{eta1} and of \eqref{eta2} also appear in the Lambert series decomposition for the forms of discriminant $-15$. We derive these Lambert series decompositions by utilizing
\begin{equation}
\label{lsp}
\phi(q)\phi(q^{15})=P(-q)-Q(-q),
\end{equation}
and
\begin{equation}
\label{qsp}
\phi(q^3)\phi(q^{5})=P(-q)+Q(-q).
\end{equation}
Equations \eqref{lsp} and \eqref{qsp} are given in Theorem 5.1 of \cite{berk1}. It is easy to show
\begin{equation}
\label{p0}
P_{2,0}~\phi(q)\phi(q^{15})=(1,1,4,q^4),
\end{equation}
and
\begin{equation}
\label{p00}
P_{2,0}~\phi(q^3)\phi(q^5)=(2,1,2,q^4).
\end{equation}
Taking $P_{2,0}$ of both sides of \eqref{lsp}, and letting $q^4 \to q$, gives
\begin{equation}
\label{pr}
(1,1,4,q)=1+\sum_{n=1}^\infty \left(\df{-15}{n}\right)\df{q^{n}}{1-q^{n}} +\sum_{n=1}^\infty \left(\df{5}{n}\right)\df{q^n -q^{2n}}{1-q^{3n}},
\end{equation}
where we used \eqref{m1}, \eqref{q2}, and \eqref{p0}. Similarly we take $P_{2,0}$ of both sides of \eqref{qsp}, and letting $q^4 \to q$, yields
\begin{equation}
\label{pr2}
(2,1,2,q)=1+\sum_{n=1}^\infty \left(\df{-15}{n}\right)\df{q^{n}}{1-q^{n}} -\sum_{n=1}^\infty \left(\df{5}{n}\right)\df{q^n -q^{2n}}{1-q^{3n}},
\end{equation}
where we used \eqref{m1}, \eqref{q2}, and \eqref{p00}. Equations \eqref{pr} and \eqref{pr2} are consistent with the results of \cite{toh} as well as \eqref{cons}.

We have now shown the theta series associated with the forms $(1,1,4)$ and $(2,1,2)$ are not only a sum of Lambert series, but also a sum of eta-quotients. Moreover, adding and subtracting \eqref{pr} and \eqref{pr2}, and employing \eqref{eta1} and \eqref{eta2}, yields
\begin{equation}
\label{pr3}
\df{E^2(q^3)E^{2}(q^{5})}{E(q)E(q^{15})}=\df{(1,1,4,q)+(2,1,2,q)}{2}, \mbox{    \hspace{1cm}      }q\df{E^2(q)E^{2}(q^{15})}{E(q^3)E(q^5)}=\df{(1,1,4,q)-(2,1,2,q)}{2}.
\end{equation}

\section{Proof of Theorem \ref{idoo}}
\label{main}

In this section we prove Theorem \ref{idoo} for all known idoneal discriminants. Our method of proof relies on the fact that there are only finitely many idoneal discriminants. Indeed, $\Delta$ is idoneal if and only if CL$(\Delta) \cong \mathbb{Z}_1$ or CL$(\Delta) \cong \mathbb{Z}_2^{r}$, $1\leq r\leq 4$, and all such discriminants with at most two exceptions are given in \cite{voight}. It is now a finite process to see which primes $p$ and which discriminants $\Delta$, have the property that $\Delta$ and $\Delta p^2$ are both idoneal. One can employ \eqref{numgenform} and \eqref{hrel} to see that if $\Delta$ and $\Delta p^2$ are both idoneal, then $p=2,3,5,7$.\\

The below table lists all known discriminants $\Delta$ and primes $p$ such that both $\Delta$ and $\Delta p^2$ are idoneal.

\begin{center}
\begin{tabular}{ | l | l | } 
\hline               
   $p$ & $-\Delta$ \\ \hline 
	$2$ & $3, 4, 7, 8, 12, 15, 16, 24, 28, 40, 48, 60, 72, 88, 112, 120,$ \\ \cline{2-2}
	& $168, 232, 240, 280, 312, 408, 520, 760, 840, 1320, 1848$\\ \hline
	$3$ & $3, 4, 8, 11, 20, 32, 35$  \\ \hline 
	$5$ & $3, 4$ \\ \hline 
	$7$ & $3$ \\ \hline
\end{tabular}
\begin{flushright}
			 .
			 \end{flushright}
   \end{center}
All identities coming from Theorem \ref{idoo} have a very similar proof structure. Furthermore, the proofs of these identities are elementary in the sense that one does not need to employ modular form techniques. We illustrate the proof technique by giving the explicit proof for $(\Delta, p) = (-3,3), (-4,5),(-3,7),(-28,2)$.


We now prove Theorem \ref{idoo} for the case $\Delta=-3$ and $p=3$. Theorem \ref{idoo} gives
\begin{equation}
3(1,1,7,q)=3(1,1,1,q^{9})+P_{3,1}(1,1,1,q).
\end{equation}

\begin{proof}
We use $x\equiv y$ to mean $x\equiv y \imod{3}$.\\
 We see that 
\begin{align*}
P_{3,1}(1,1,1,q)&=\sum_{\substack{x+y\equiv 0\\  y\not\equiv 0}}q^{x^{2}+xy+y^{2}}+\sum_{\substack{x\equiv 0\\  y\not\equiv 0}}q^{x^{2}+xy+y^{2}}+\sum_{\substack{x\not\equiv 0\\  y\equiv 0}}q^{x^{2}+xy+y^{2}}\\
&=\sum_{x+y\equiv 0}q^{x^{2}+xy+y^{2}}+\sum_{\substack{x\equiv 0\\  y}}q^{x^{2}+xy+y^{2}}+\sum_{\substack{x\\  y\equiv 0}}q^{x^{2}+xy+y^{2}}-3\sum_{\substack{x\equiv 0\\  y\equiv 0}}q^{x^{2}+xy+y^{2}}\\
&=3\sum_{x,y}q^{x^{2}+3xy+9y^{2}}-3(1,1,1,q^{9})\\
&=3\sum_{x,y}q^{(x-y)^{2}+3(x-y)y+9y^{2}}-3(1,1,1,q^{9})\\
&=3(1,1,7,q)-3(1,1,1,q^{9}).
\end{align*}
\end{proof}

We prove Theorem \ref{idoo} for the case $\Delta=-4$ and $p=5$. Theorem \ref{idoo} gives
\begin{align}
2(1,0,25,q)&=2(1,0,1,q^{25})+(P_{5,1}+P_{5,4})(1,0,1,q),\label{iiv}\\
2(2,2,13,q)&=2(1,0,1,q^{25})+(P_{5,2}+P_{5,3})(1,0,1,q).\label{iiiv}
\end{align}

\begin{proof}
For this proof we use $x\equiv y$ to mean $x\equiv y \imod{5}$. We see that 
\begin{align*}
(P_{5,1}+P_{5,4})(1,0,1,q)&=\sum_{\substack{x\equiv 0\\  y\not\equiv 0}}q^{x^{2}+y^{2}}+\sum_{\substack{x\not\equiv 0\\  y\equiv 0}}q^{x^{2}+y^{2}}\\
&=2\sum_{\substack{x\equiv 0\\  y}}q^{x^{2}+y^{2}}-2\sum_{\substack{x\equiv 0\\  y\equiv 0}}q^{x^{2}+y^{2}}\\
&=2(1,0,25,q)-2(1,0,1,q^{25}),
\end{align*}
and have established \eqref{iiv}.\\
Similarly,
\begin{align*}
(P_{5,2}+P_{5,3})(1,0,1,q)&=\sum_{\substack{x\equiv y\\  x\not\equiv 0}}q^{x^{2}+y^{2}}+\sum_{\substack{x\equiv -y\\  x\not\equiv 0}}q^{x^{2}+y^{2}}\\
&=\sum_{x\equiv y}q^{x^{2}+y^{2}}+\sum_{x\equiv -y}q^{x^{2}+y^{2}}-2\sum_{\substack{x\equiv 0\\  y\equiv 0}}q^{x^{2}+y^{2}}\\
&=2\sum_{x\equiv y}q^{x^{2}+y^{2}}-2\sum_{\substack{x\equiv 0\\  y\equiv 0}}q^{x^{2}+y^{2}}\\
&2\sum_{x,y}q^{(x+3y)^{2}+(x-2y)^{2}}-2\sum_{\substack{x\equiv 0\\  y\equiv 0}}q^{x^{2}+y^{2}}\\
&=2(2,2,13,q)-2(1,0,1,q^{25}),
\end{align*}
which completes the proof.
\end{proof}

Theorem \ref{idoo} for the case $\Delta=-3$ and $p=7$ gives
\begin{align}
3(1,1,37,q)&=3(1,1,1,q^{49})+(P_{7,1}+P_{7,2}+P_{7,4})(1,1,1,q),\label{147p}\\
3(3,3,13,q)&=3(1,1,1,q^{49})+(P_{7,3}+P_{7,5}+P_{7,6})(1,1,1,q).\label{147pp}
\end{align}
\begin{proof}
As expected, we use $x\equiv y$ to mean $x\equiv y \imod{7}$.\\
 We see that 
\begin{align*}
(P_{7,1}+P_{7,2}+P_{7,4})(1,1,1,q)&=\sum_{\substack{x+y\equiv 0\\  y\not\equiv 0}}q^{x^{2}+xy+y^{2}}+\sum_{\substack{x\equiv 0\\  y\not\equiv 0}}q^{x^{2}+xy+y^{2}}+\sum_{\substack{x\not\equiv 0\\  y\equiv 0}}q^{x^{2}+xy+y^{2}}\\
&=\sum_{x+y\equiv 0}q^{x^{2}+xy+y^{2}}+\sum_{\substack{x\equiv 0\\  y}}q^{x^{2}+xy+y^{2}}+\sum_{\substack{y\equiv 0\\  x}}q^{x^{2}+xy+y^{2}}-3\sum_{\substack{x\equiv 0\\  y\equiv 0}}q^{x^{2}+xy+y^{2}}\\
&=3\sum_{x,y}q^{x^{2}+7xy+49y^{2}}-3(1,1,1,q^{49})\\
&=3\sum_{x,y}q^{(x-3y)^{2}+7(x-3y)y+49y^{2}}-3(1,1,1,q^{49})\\
&=3(1,1,37,q)-3(1,1,1,q^{49}).
\end{align*}
Similarly
\begin{align*}
(P_{7,3}+P_{7,5}+P_{7,6})(1,1,1,q)&=\sum_{\substack{x\equiv y\\  y\not\equiv 0}}q^{x^{2}+xy+y^{2}}+\sum_{\substack{x\equiv 3y\\  y\not\equiv 0}}q^{x^{2}+xy+y^{2}}+\sum_{\substack{x\equiv 5y\\  y\not\equiv 0}}q^{x^{2}+xy+y^{2}}\\
&=3\sum_{x,y}q^{3x^{2}+3xy+13y^{2}}-3(1,1,1,q^{49})\\
&=3(3,3,13,q)-3(1,1,1,q^{49}).
\end{align*}
\end{proof}

The last example we consider has $\Delta$ being a nonfundamental discriminant. This does not complicate matters, and the proof technique is still similar to the examples we have already seen. Applying Theorem \eqref{idoo} with $\Delta=-112$ and $p=2$ gives
\begin{align}
(1,0,112,q)&=(1,0,28,q^{4})+P_{8,1}(1,0,28,q),\\
(4,4,29,q)&=(1,0,28,q^{4})+P_{8,5}(1,0,28,q),\\
(7,0,16,q)&=(4,0,7,q^{4})+P_{8,7}(4,0,7,q),\\
(11,6,11,q)&=(4,0,7,q^{4})+P_{8,3}(4,0,7,q).
\end{align}

\begin{proof}
We use $x \equiv y$ to mean $x \equiv y \imod{2}$. We have
\begin{align*}
P_{8,1}(1,0,28,q)&=\sum_{\substack{y\equiv 0\\  x\not\equiv 0}}q^{x^{2}+28y^{2}}\\
&=\sum_{\substack{y\equiv 0\\  x}}q^{x^{2}+28y^{2}}-\sum_{\substack{y\equiv 0\\  x\equiv 0}}q^{x^{2}+28y^{2}}\\
&=(1,0,112,q)-(1,0,28,q^{4}),
\end{align*}
\begin{align*}
P_{8,5}(1,0,28,q)&=\sum_{\substack{x+y\equiv 0\\  y\not\equiv 0}}q^{x^{2}+28y^{2}}\\
&=\sum_{x+y\equiv 0}q^{x^{2}+28y^{2}}-\sum_{\substack{y\equiv 0\\  x\equiv 0}}q^{x^{2}+28y^{2}}\\
&=\sum_{x,y}q^{(2x+y)^{2}+28y^{2}}-(1,0,28,q^{4})\\
&=(4,4,29,q)-(1,0,28,q^{4}),
\end{align*}
\begin{align*}
P_{8,7}(4,0,7,q)&=\sum_{\substack{x\equiv 0\\  y\not\equiv 0}}q^{4x^{2}+7y^{2}}\\
&=\sum_{\substack{x\equiv 0\\  y}}q^{4x^{2}+7y^{2}}-\sum_{\substack{y\equiv 0\\  x\equiv 0}}q^{4x^{2}+7y^{2}}\\
&=(7,0,16,q)-(4,0,7,q^{4}),
\end{align*}
and
\begin{align*}
P_{8,3}(4,0,7,q)&=\sum_{\substack{x+y\equiv 0\\  y\not\equiv 0}}q^{4x^{2}+7y^{2}}\\
&=\sum_{x+y\equiv 0}q^{4x^{2}+7y^{2}}-\sum_{\substack{y\equiv 0\\  x\equiv 0}}q^{4x^{2}+7y^{2}}\\
&=\sum_{x,y}q^{4(x-y)^{2}+7(x+y)^{2}}-(4,0,7,q^{4})\\
&=(11,6,11,q)-(4,0,7,q^{4}).
\end{align*}

\end{proof}

\section{Explicit Representation Formulas}
\label{example}

In this section we give detailed examples in which we employ Theorem \ref{idoo} to derive explicit formulas for the number of representations by binary quadratic forms. It is easily seen that the method illustrated in this section applies to any idoneal nonfundamental discriminant.\\

We start with the simple example of discriminant $-36$, which has class group isomorphic to $\mathbb{Z}_2$. We use Theorem \ref{idoo} to derive the Lambert series decomposition for $(1,0,9,q)$ and $(2,2,5,q)$, and subsequently product representation formulas for the corresponding forms. We remark that \cite{toh} does not give Lambert series decompositions for these forms since $-36$ is a nonfundamental discriminant.\\
Theorem \ref{idoo} yields
\begin{equation}
\label{36o}
2(1,0,9,q)=2(1,0,1,q^9) + P_{3,1}(1,0,1,q),
\end{equation}
and
\begin{equation}
\label{36oo}
2(2,2,5,q)=2(1,0,1,q^9) + P_{3,2}(1,0,1,q).
\end{equation}
Employing the well known
\begin{equation}
\label{36p}
(1,0,1,q)=1+4\sum_{n=1}^{\infty}\left(\frac{-4}{n}\right) \frac{q^{n}}{1-q^{n}},
\end{equation}
it is not hard to show
\begin{equation}
\label{36pp}
P_{3,1}(1,0,1,q)-P_{3,2}(1,0,1,q)=4\sum_{n=1}^{\infty}\left(\frac{12}{n}\right) \frac{q^{n}(1-q^{n})}{1-q^{3n}}.
\end{equation}

Employing \eqref{36o}--\eqref{36pp} we find the Lambert series decompositions
\begin{equation}
\label{36dec1}
(1,0,9,q) = 1+\sum_{n=1}^{\infty}\left(\frac{-4}{n}\right) \frac{q^{n}}{1-q^{n}} +3\sum_{n=1}^{\infty}\left(\frac{-4}{n}\right) \frac{q^{9n}}{1-q^{9n}}+\sum_{n=1}^{\infty}\left(\frac{12}{n}\right) \frac{q^{n}(1-q^{n})}{1-q^{3n}},
\end{equation}
and
\begin{equation}
\label{36dec1}
(2,2,5,q)= 1+\sum_{n=1}^{\infty}\left(\frac{-4}{n}\right) \frac{q^{n}}{1-q^{n}} +3\sum_{n=1}^{\infty}\left(\frac{-4}{n}\right) \frac{q^{9n}}{1-q^{9n}}-\sum_{n=1}^{\infty}\left(\frac{12}{n}\right) \frac{q^{n}(1-q^{n})}{1-q^{3n}}.
\end{equation}

We have
\[
A(n):=[q^n]\sum_{n=1}^{\infty}\left(\frac{-4}{n}\right) \frac{q^{n}}{1-q^{n}}=\sum_{d|n}\left(\frac{-4}{d}\right).
\]
Also we see
\begin{align*}
\sum_{n=1}^{\infty}\left(\frac{12}{n}\right) \frac{q^{n}(1-q^{n})}{1-q^{3n}} &= \sum_{n=1}^{\infty}\sum_{m=0}^{\infty}\left(\frac{12}{n}\right)(q^{n(3m+1)} - q^{n(3m+2)})\\
&=\sum_{n=1}^{\infty}\sum_{m=1}^{\infty}\left(\frac{12}{n}\right)(q^{n(3m-1)} - q^{n(3m-2)})\\
&=\sum_{n=1}^{\infty}\sum_{m=1}^{\infty}\left(\frac{12}{n}\right)\left(\frac{m}{3}\right)q^{nm}\\
&=\sum_{n=1}^{\infty}\left(\sum_{d|n}\left(\frac{12}{d}\right)\left(\frac{n/d}{3}\right) \right)q^{n},
\end{align*}
so that
\[
D(n):=[q^n]\sum_{n=1}^{\infty}\left(\frac{12}{n}\right) \frac{q^{n}(1-q^{n})}{1-q^{3n}}=\sum_{d|n}\left(\frac{12}{d}\right)\left(\frac{n/d}{3}\right).
\]

Since $A(n)$ and $D(n)$ are multiplicative, it suffices to find their values at prime powers. It is easy to check that for a prime $p$
\[
A(p^{\alpha}) = \left\{ \begin{array}{ll}
       1 &  p=2, \\
       1+\alpha &  p \equiv 1 \imod{4}, \\
       \frac{(-1)^{\alpha}+1}{2} &  p \equiv 3 \imod{4},
     \end{array}
     \right.
\]
and
\[
D(p^{\alpha}) = \left\{ \begin{array}{ll}
       0 &  p=3, \alpha \neq 0 \\
       (-1)^{\alpha} &  p=2,\\
       1+\alpha &  p \equiv 1 \imod{12}, \\
       (-1)^{\alpha}(1+\alpha) &  p \equiv 5 \imod{12}, \\
       \frac{(-1)^{\alpha}+1}{2} &  p \equiv 7,11 \imod{12}.
     \end{array}
     \right.
\]
\vspace{.3cm}

The product representation formulas for the forms of discriminant $-36$ are given in the following theorem.

\begin{theorem}
Let the prime factorization of $n$ be
\[
n=2^{a}3^{b}\prod_{i=1}^{r}p_{i}^{v_{i}}\prod_{j=1}^{s}q_{j}^{w_{j}},
\]
where $p_{i} \equiv 1 \imod{4}$ and $3<q_{j}\equiv 3\imod{4}$. We define
\[
\Lambda(n):=\prod_{i=1}^{r}(1+{v_{i}})\prod_{j=1}^{s}\tf{1+(-1)^{w_{j}}}{2}.
\]
We find the representation formula for $(1,0,9)$ to be
\begin{equation}
\label{h}
(1,0,9;n) = \left\{ \begin{array}{ll}
        (1+(-1)^{a+t})\Lambda(n)&  b=0,\\
				0&  b\equiv 1 \imod{2},\\
				4\Lambda(n)&  0<b\equiv 0 \imod{2},\\
     \end{array}
     \right.
\end{equation}
and the representation formula for $(2,2,5)$ to be
\begin{equation}
\label{h2}
(2,2,5;n) = \left\{ \begin{array}{ll}
        (1-(-1)^{a+t})\Lambda(n)&  b=0,\\
				0&  b\equiv 1 \imod{2},\\
				4\Lambda(n)&  0<b\equiv 0 \imod{2}.\\
     \end{array}
     \right.
\end{equation}
where $t$ is the number of prime factors of $n$, counting multiplicity, that are congruent to $5\imod{12}$.
\end{theorem}

We remark that \eqref{h} and \eqref{h2} imply $(1,0,9;n)=(2,2,5;n)$ if $3\mid n$, and $(1,0,9;n)\cdot(2,2,5;n)=0$ if $3\nmid n$.


The next example gives the Lambert series decomposition as well as the representation formulas for the forms of discriminant $-75$.\\

For $\Delta=-3$ and $p=5$, Theorem \ref{idoo} gives
\begin{equation}
\label{71w}
3(1,1,19,q)=3(1,1,1,q^{25})+(P_{5,1}+P_{5,4})(1,1,1,q),
\end{equation}
and
\begin{equation}
\label{71ww}
3(3,3,7,q)=3(1,1,1,q^{25})+(P_{5,2}+P_{5,3})(1,1,1,q).
\end{equation}

Combining \eqref{71w} and \eqref{71ww} we have,
\begin{equation}
\label{oee}
\begin{aligned}
3(1,1,19,q)+3(3,3,7,q)&=6(1,1,1,q^{25})+(1,1,1,q)-P_{5,0}(1,1,1,q)\\
&=(1,1,1,q)+5(1,1,1,q^{25}),
\end{aligned}
\end{equation}
and
\begin{equation}
\label{oe}
3(1,1,19,q)-3(3,3,7,q)=(P_{5,1}+P_{5,4}-P_{5,2}-P_{5,3})(1,1,1,q).
\end{equation}

Because of the relations
\begin{equation}
\label{75t}
(1,1,1,q) = 1+6\sum_{n=1}^{\infty}\left(\frac{-3}{n}\right)\frac{q^{n}}{1-q^{n}},
\end{equation}
and
\begin{equation}
\label{75tt}
(P_{5,1}+P_{5,4}-P_{5,2}-P_{5,3})(1,1,1,q) = 6\sum_{n>0}\left(\frac{-15}{n}\right)\frac{q^{n}-q^{2n}-q^{3n}+q^{4n}}{1-q^{5n}},
\end{equation}
we are led to define 
\[
f(q)=\sum_{n>0}\left(\frac{-3}{n}\right)\frac{q^{n}}{1-q^{n}},
\]
and
\[
g(q) = \sum_{n>0}\left(\frac{-15}{n}\right)\frac{q^{n}-q^{2n}-q^{3n}+q^{4n}}{1-q^{5n}}.
\]

Combining \eqref{oee}--\eqref{75tt} we find
\begin{equation}
\label{k}
(1,1,19,q)+(3,3,7,q)=2+2f(q)+10f(q^{25}),
\end{equation}
and
\begin{equation}
\label{kk}
(1,1,19,q)-(3,3,7,q)=2g(q).
\end{equation}

Adding and subtracting \eqref{k} and \eqref{kk} gives the Lambert series decompositions
\begin{equation}
\label{kkk}
(1,1,19,q)=1+f(q)+5f(q^{25})+g(q),
\end{equation}
and
\begin{equation}
\label{kkkkk}
(3,3,7,q)=1+f(q)+5f(q^{25})-g(q).
\end{equation}

We note that
\[
[q^n]f(q)=\sum_{d|n}\left(\frac{-3}{d}\right),
\]
\[
[q^n]g(q)=\sum_{d|n}\left(\frac{-15}{d}\right)\left(\frac{5}{n/d}\right),
\]
are multiplicative and we find the values
\[
[q^{p^{\alpha}}]f(q) = \left\{ \begin{array}{ll}
       1 &  p=3, \\
       1+\alpha &  p \equiv 1 \imod{3}, \\
       \frac{(-1)^{\alpha}+1}{2} &  p \equiv 2 \imod{3},
     \end{array}
     \right.
\]
and
\[
[q^{p^{\alpha}}]g(q) = \left\{ \begin{array}{ll}
       0 &p=5,(\alpha > 0),\\
       (-1)^{\alpha} &  p=3, \\
       1+\alpha &  p \equiv 1,4 \imod{15}, \\
       (-1)^{\alpha}(1+\alpha) &  p \equiv 7,13 \imod{15}, \\
       \frac{(-1)^{\alpha}+1}{2} &  5\neq p \equiv 2 \imod{3}.
     \end{array}
     \right.
\]

The representation formulas for $(1,1,19)$ and $(3,3,7)$ are given in the following theorem.

\begin{theorem}
Let the prime factorization of $n$ be
\[
n=3^{a}5^{b}\prod_{i=1}^{r}p_{i}^{v_{i}}\prod_{j=1}^{s}q_{j}^{w_{j}},
\]
where $p_{i}\equiv 1 \imod{3}$ and $5\neq q_{j}\equiv 2 \imod{3}$. Let
\[
\Lambda(n):=\prod_{i=1}^{r}(1+{v_{i}})\prod_{j=1}^{s}\frac{1+(-1)^{w_{j}}}{2}.
\]
The representation formula for $(1,1,19)$ is
\begin{equation}
\label{h5h}
(1,1,19;n) = \left\{ \begin{array}{ll}
        (1+(-1)^{a+t})\Lambda(n)&  b=0,\\
				0&  b\equiv 1 \imod{2},\\
				6\Lambda(n)&  0<b\equiv 0 \imod{2},\\
     \end{array}
     \right.
\end{equation}
and the formula for $(3,3,7;n)$ is 
\begin{equation}
\label{h5hh}
(3,3,7;n) = \left\{ \begin{array}{ll}
        (1-(-1)^{a+t})\Lambda(n)&  b=0,\\
				0&  b\equiv 1 \imod{2},\\
				6\Lambda(n)&  0<b\equiv 0 \imod{2},\\
     \end{array}
     \right.
\end{equation}
where $t$ is the number of prime factors of $n$, counting multiplicity, that are congruent to $7,13 \imod{15}$.
\end{theorem}
Note that \eqref{h5h} and \eqref{h5hh} imply $(1,1,19;n)=(3,3,7;n)$ if $5\mid n$, and $(1,1,19;n)\cdot(3,3,7;n)=0$ if $5\nmid n$.



The next example discusses $\Delta=-20$ and $p=3$. The genus structure for the relevant discriminants are given in Example \ref{excor}. We note this example has CL$(\Delta)>1$ and Theorem \ref{idoo} yields 4 identities. Employing Theorem \ref{idoo} we have

\begin{align}
(1,0,45,q)&=(1,0,5,q^{9})+P_{3,1}(1,0,5,q),\label{aa1}\\
(5,0,9,q)&=(1,0,5,q^{9})+P_{3,2}(1,0,5,q),\label{aa2}\\
(7,4,7,q)&=(2,2,3,q^{9})+P_{3,1}(2,2,3,q),\label{aa3}\\
(2,2,23,q)&=(2,2,3,q^{9})+P_{3,2}(2,2,3,q)\label{aa4}.
\end{align}
Let
\begin{align*}
P(q)&:=\sum_{n=1}^{\infty}\left(\frac{-20}{n}\right)\frac{q^{n}}{1-q^{n}},\\
Q(q)&:=\sum_{n=1}^{\infty}\left(\frac{n}{5}\right)\frac{q^{n}}{1+q^{2n}}.
\end{align*}

The Lambert series decompositions for $(1,0,5)$ and $(2,2,3)$ are given in \cite[(3.3), (3.10)]{berk1} and are
\begin{equation}
\label{dec1}
(1,0,5,q) = 1+P(q)+Q(q),
\end{equation}
and
\begin{equation}
\label{dec2}
(2,2,3,q) = 1+P(q)-Q(q).
\end{equation}

We note that \eqref{dec1} and \eqref{dec2} are in agreement with \eqref{cons}. Employing \eqref{dec1} and \eqref{dec2} we find
\begin{equation}
\label{dec3}
(P_{3,1}-P_{3,2})(1,0,5,q) = R(q)+S(q),
\end{equation}
and
\begin{equation}
\label{dec4}
(P_{3,1}-P_{3,2})(2,2,3,q) = R(q)-S(q),
\end{equation}
where
\begin{align*}
R(q)&:=\sum_{n=1}^{\infty}\left(\frac{60}{n}\right)\frac{q^{n}-q^{2n}}{1-q^{3n}},\\
S(q)&:=\sum_{n=1}^{\infty}\left(\frac{n}{15}\right)\frac{q^{n}-q^{5n}}{1+q^{6n}}.
\end{align*}

We now have all the Lambert series necessary to proceed in a similar manner to the previous examples of this section. Employing \eqref{aa1}--\eqref{dec4} we find
\begin{align}
2(1,0,45,q) &= 2+P(q)+Q(q)-2P(q^{3})+2Q(q^{3})+3P(q^{9})+3Q(q^{9})+R(q)+S(q), \label{ab1}\\
2(5,0,9,q) &= 2+P(q)+Q(q)-2P(q^{3})+2Q(q^{3})+3P(q^{9})+3Q(q^{9})-R(q)-S(q),\label{ab2}\\
2(7,4,7,q) &= 2+P(q)-Q(q)-2P(q^{3})-2Q(q^{3})+3P(q^{9})-3Q(q^{9})+R(q)-S(q),\label{ab3}\\
2(2,2,23,q) &= 2+P(q)-Q(q)-2P(q^{3})-2Q(q^{3})+3P(q^{9})-3Q(q^{9})-R(q)+S(q).\label{ab4}
\end{align}

Equations \eqref{ab1}--\eqref{ab4} give the Lambert series decomposition for the forms of discriminant $-180$. We find

\[
[q^{p^\alpha}]P(q)=\sum_{d|p^{\alpha}}\left(\frac{-20}{d}\right)= \left\{ \begin{array}{ll}
       1 &  p=2,5, \\
       1+\alpha &  p \equiv 1,3,7,9 \imod{20}, \\
       \frac{(-1)^{\alpha}+1}{2} &  p \equiv 11,13,17,19 \imod{20},
     \end{array}
     \right.
\]
\[
[q^{p^\alpha}]Q(q) =\sum_{d|p^{\alpha}}\left(\frac{-4}{d}\right)\left(\frac{p^{\alpha}/d}{5}\right)= \left\{ \begin{array}{ll}
(-1)^{\alpha} &  p=2, \\
       1 &  p=5, \\
       1+\alpha &  p \equiv 1,9 \imod{20},\\
       (-1)^{\alpha}(1+\alpha) &  p \equiv 3,7 \imod{20},\\
       \frac{(-1)^{\alpha}+1}{2} &  p \equiv 11,13,17,19 \imod{20},
     \end{array}
     \right.
\]

\[
[q^{p^\alpha}]R(q)=\sum_{d|p^{\alpha}}\left(\frac{60}{d}\right)\left(\frac{-3}{p^{\alpha}/d}\right) = \left\{ \begin{array}{ll}
       0 &p=3, (\alpha>0),\\
       (-1)^{\alpha} &  p=2,5,\\
       1+\alpha &  p \equiv 1,7,43,49 \imod{60}, \\
       (-1)^{\alpha}(1+\alpha) &  p \equiv 23,29,41,47 \imod{60}, \\
       \frac{(-1)^{\alpha}+1}{2} & p \equiv 11,13,17,19 \imod{20},
     \end{array}
     \right.
\]
	
\[
[q^{p^\alpha}]S(q)=\sum_{d|p^{\alpha}}\left(\frac{d}{15}\right)\left(\frac{12}{p^{\alpha}/d}\right)  = \left\{ \begin{array}{ll}
       1 &p=2,\\
       0 &p=3, (\alpha>0),\\
       (-1)^{\alpha} &p=5,\\
       1+\alpha &  p \equiv 1,23,47,49 \imod{60}, \\
       (-1)^{\alpha}(1+\alpha) &  p \equiv 7,29,41,43 \imod{60}, \\
       \frac{(-1)^{\alpha}+1}{2} &  p \equiv 11,13,17,19 \imod{20}.
     \end{array}
     \right.
\]
We now deduce the representation formulas for the forms of discriminant $-180$.

\begin{theorem}
Let the prime factorization of $n$ be given by
\[
n=2^{a}3^{b}5^{c}\prod_{i=1}^{y}p_{i}^{v_{i}}\prod_{j=1}^{s}q_{j}^{w_{j}}
\]
where $3< p_{i}\equiv 1,3,7,9 \imod{20}$, $q_{j}\equiv 11,13,17,19 \imod{20}$. Let
\[
\Lambda(n):=\prod_{i=1}^{y}(1+{v_{i}})\prod_{j=1}^{s}\frac{1+(-1)^{w_{j}}}{2}.
\]
The representation formulas for the forms of discriminant $-180$ are
\begin{equation}
\label{h6h}
(1,0,45;n) = \left\{ \begin{array}{ll}
        \tf{1}{2}(1+(-1)^{a+t_{1}}+(-1)^{a+c+t_{2}}+(-1)^{c+t_{3}})\Lambda(n)&  b=0,\\
				(b-1)(1+(-1)^{a+t_{1}})\Lambda(n)&  b>0,\\
     \end{array}
     \right.
\end{equation}

\begin{equation}
\label{h6h2}
(5,0,9;n) = \left\{ \begin{array}{ll}
        \tf{1}{2}(1+(-1)^{a+t_{1}}-(-1)^{a+c+t_{2}}-(-1)^{c+t_{3}})\Lambda(n)&  b=0,\\
				(b-1)(1+(-1)^{a+t_{1}})\Lambda(n)&  b>0,\\
     \end{array}
     \right.
\end{equation}

\begin{equation}
\label{h6h22}
(7,4,7;n) = \left\{ \begin{array}{ll}
        \tf{1}{2}(1-(-1)^{a+t_{1}}+(-1)^{a+c+t_{2}}-(-1)^{c+t_{3}})\Lambda(n)&  b=0,\\
				(b-1)(1-(-1)^{a+t_{1}})\Lambda(n)&  b>0,\\
     \end{array}
     \right.
\end{equation}

\begin{equation}
\label{h6h222}
(2,2,23;n) = \left\{ \begin{array}{ll}
        \tf{1}{2}(1-(-1)^{a+t_{1}}-(-1)^{a+c+t_{2}}+(-1)^{c+t_{3}})\Lambda(n)&  b=0,\\
				(b-1)(1-(-1)^{a+t_{1}})\Lambda(n)&  b>0,\\
     \end{array}
     \right.
\end{equation}
where $t_{1}, t_2$, and $t_3$ are the number of prime factors of $n$, counting multiplicity, that are congruent to $3,7 \imod{20}$, $23,29,41,47\imod{60}$, or $7,29,41,43\imod{60}$ respectively.

\end{theorem}


We remark that the right hand side of \eqref{h6h}--\eqref{h6h222} is always nonnegative, because if the coefficient of $\Lambda(n)$ is negative, then we must have $\Lambda(n)=0$. The examples of this section illustrate how Theorem \ref{idoo} can be used to write the theta series of a form of nonfundamental discriminant as a linear combination of Lambert series. The utility of writing the theta series as a linear combination of Lambert series, is the ability to write a product formula for the total number of representations by the form.

\section{The general identity}
\label{main2}

In \cite[pg.119]{buell}, Buell gives a representation of all primitive forms of discriminant $\Delta p^2$ in terms of the primitive forms of discriminant $\Delta$. Explicitly, all primitive forms of discriminant $\Delta p^2$ are given by the primitive forms contained in
\begin{equation}
\label{buellh}
\begin{aligned}
\{(a,bp,cp^2)\} \cup \{(ap^2, pb+2ahp, ah^2 +bh+c): 0\leq h<p \},
\end{aligned}
\end{equation}
where $(a,b,c)$ is a primitive form of discriminant $\Delta$. Building on this observation, we define $\Psi_{G,p}(a,b,c)$ to be the set of primitive forms in \eqref{buellh} which are also in the genus $G$. We now introduce a new theorem.

\begin{theorem}
	\label{newthm}
	 Let $(a,b,c)$ be a primitive form of discriminant $\Delta$, and $G$ a genus of discriminant $\Delta p^2$ with $\Psi_{G,p}(a,b,c)$ nonempty. For $p$ an odd prime, we have
\[
 w\sum_{(A,B,C)\in \Psi_{G,p}(a,b,c)}(A,B,C,q)
=   w|\Psi_{G,p}(a,b,c)|(a,b,c,q^{p^{2}}) +\sum_{i=1}^{p-1} \frac{1}{2}( \left(\frac{ri}{p}\right) +1)P_{p,i} (a,b,c,q)
\]

 and for $p=2$,
\[
 w\sum_{(A,B,C)\in \Psi_{G,2}(a,b,c)}(A,B,C,q)
=  w|\Psi_{G,2}(a,b,c)|(a,b,c,q^{4}) + P_{2^{t+1},r} (a,b,c,q) 
\]
\noindent
where\\
\[
	w :=\left\{ \begin{array}{ll}
        3&  \Delta  =-3,\\
				   2&  \Delta =-4,\\
					   1&  \Delta  <-4,\\
     \end{array}
     \right.
	\]
	$r$ is coprime to $\Delta p^2$ and is represented by any form of $\Psi_{G,p}(a,b,c)$. When $\Delta \equiv 0 \imod{16}$ we define $t=2$, and for $\Delta \not\equiv 0 \imod{16}$ we define $t=0,1$ according to whether $\Delta$ is odd or even.
\end{theorem}
It is not hard to see that Theorem \ref{idoo} follows from Theorem \ref{newthm}. The proof of Theorem \ref{newthm} will be given in a subsequent paper. For now we content ourselves with an illustrative example of Theorem \ref{newthm}.

We consider the nonfundamental discriminant $\Delta=-92$ which has class group isomorphic to $\mathbb{Z}_3$. The forms are given in the table

\begin{center}
\begin{tabular}{ | l | l | l |}
  \hline     
  \multicolumn{2}{|c|}{CL$(-92) \cong \mathbb{Z}_3$}& $\left(\tf{r}{23}\right)$ \\
  \hline                   
   $g$ & $(1,0,23)$, $(3,\pm 2,8)$ &$+1$ \\ \hline 
\end{tabular}
\begin{flushright}
			 .
			 \end{flushright}
   \end{center}
	
We take $p=5$	and hence $\Delta p^2=-2300$. The forms and class group group structure for $\Delta p^2$ are given in the following table
	
\begin{center}
\begin{tabular}{ | l | l | l | l | }
  \hline     
  \multicolumn{2}{|c|}{CL$(-2300) \cong \mathbb{Z}_{18}$}& $\left(\tf{r}{5}\right)$ & $\left(\tf{r}{23}\right)$\\
  \hline                   
   $G_1$ & $(1,0,575)$, $(9,\pm 2,64)$, $(16,\pm 14,39)$, $(24,\pm 22,29)$, $(24,\pm 10,25)$&$+1$ & $+1$ \\ \hline 
	$G_2$ & $(23,0,25)$, $(3,\pm 2,192)$, $(8,\pm 6,73)$, $(13,\pm 12,47)$, $(25,\pm 20,27)$ &$-1$ & $+1$ \\ \hline 
\end{tabular}
\begin{flushright}
			 .
			 \end{flushright}
   \end{center}
	
	We compute
	\begin{equation}
	\label{psihelp}
	\begin{aligned}
	\Psi_{G_1,5}(1,0,23)&=\{(1,0,575), (24,\pm 10,25)\},\\
	\Psi_{G_2,5}(1,0,23)&=\{(23,0,25), (25,\pm 20,27)\},\\
	\Psi_{G_1,5}(3,2,8)&=\{(9,2,64), (16,-14,39), (24,-22,29)\},\\
	\Psi_{G_2,5}(3,2,8)&=\{(3,-2,192), (8,6,73), (13,12,47)\},
	\end{aligned}
	\end{equation}
	where $G_1$ and $G_2$ are the genera of discriminant $-2300$ given in the above table. Using \eqref{psihelp} we see that for $\Delta=-92$ and $p=5$, Theorem \ref{newthm} yields the identities
\begin{align}
(1,0,575,q) + 2(24,10,25,q) &= 3(1,0,23,q^{25}) + (P_{5,1}+P_{5,4})(1,0,23,q),\label{nt1}\\
(23,0,25,q) + 2(25,20,27,q) &= 3(1,0,23,q^{25}) + (P_{5,2}+P_{5,3})(1,0,23,q),\label{nt2}\\
(9,2,64,q) + (16,14,39,q) + (24,22,29,q) &= 3(3,2,8,q^{25}) + (P_{5,1}+P_{5,4})(3,2,8,q),\label{nt3}\\
(3,2,192,q) + (8,6,73,q)+ (13,12,47,q) &= 3(3,2,8,q^{25}) + (P_{5,2}+P_{5,3})(3,2,8,q).\label{nte}
\end{align}

Individually the identities \eqref{nt1}--\eqref{nte} do not directly yield Lambert series because each of these identities involve theta series not associated with the entire genus. However we can combine the respective identities in order to derive an identity for the theta series of the entire genus. If we add \eqref{nt1} to 2 copies of \eqref{nt3} we have
\begin{equation}
\label{co1}
 \sum_{(A,B,C)\in G_1}(A,B,C,q) = \sum_{(a,b,c)\in g}\left[3(a,b,c,q^{25}) + (P_{5,1}+P_{5,4})(a,b,c,q) \right].
\end{equation}
Similarly adding \eqref{nt2} to 2 copies of \eqref{nte} yields
\begin{equation}
\label{co2}
 \sum_{(A,B,C)\in G_2}(A,B,C,q) = \sum_{(a,b,c)\in g}\left[3(a,b,c,q^{25}) + (P_{5,2}+P_{5,3})(a,b,c,q) \right].
\end{equation}
The identities \eqref{co1} and \eqref{co2} lead us to the following corollary of Theorem \ref{newthm}.

\begin{corollary}
	\label{genflaw}
	 Let $g$ be a genus of discriminant $\Delta$ and $G$ a genus of discriminant $\Delta p^2$. If $G$ corresponds to $g$ (see Definition \ref{corr}), then for $p$ odd we have
\[
 w\sum_{(A,B,C)\in G}(A,B,C,q)
=  \sum_{(a,b,c)\in g}\left[ w\tf{|G|}{|g|}(a,b,c,q^{p^{2}}) +\sum_{i=1}^{p-1} \tf{\left(\tf{ri}{p}\right) +1}{2}P_{p,i} (a,b,c,q) \right]
\]

 and for $p=2$,
\[
 w\sum_{(A,B,C)\in G}(A,B,C,q)
=  \sum_{(a,b,c)\in g}\left[ w\tf{|G|}{|g|}(a,b,c,q^{4}) + P_{2^{t+1},r} (a,b,c,q) \right]
\]
\noindent
where\\
\[
	w :=\left\{ \begin{array}{ll}
        3&  \Delta  =-3,\\
				   2&  \Delta =-4,\\
					   1&  \Delta  <-4,\\
     \end{array}
     \right.
	\]
	$r$ is coprime to $\Delta p^2$ and is represented by $G$. When $\Delta \equiv 0 \imod{16}$ we define $t=2$, and for $\Delta \not\equiv 0 \imod{16}$ we define $t=0,1$ according to whether $\Delta$ is odd or even.
\end{corollary}
We note that if $\Delta$ is fundamental, then Corollary \ref{genflaw} directly leads to a Lambert series decomposition. If $\Delta$ is nonfundamental, then repeated applications of Corollary \ref{genflaw} will yield a Lambert series decomposition for any genus of discriminant $\Delta p^2$. Thus \eqref{co1} and \eqref{co2} do no immediately yield Lambert series, but after one application of Corollary \ref{genflaw} we would come to the Lambert series decompositions.\\
To show Corollary \ref{genflaw} follows from Theorem \ref{newthm} we first discuss some properties of the map $\Psi_{G,p}$. Indeed, one can show
\begin{equation}
\label{pro1}
|\Psi_{G,p}(f_1)| = |\Psi_{G,p}(f_2)|,
\end{equation}
and
\begin{equation}
\label{pro2}
\Psi_{G,p}(f_1) \cap \Psi_{G,p}(f_2) = \varnothing,
\end{equation}
where $f_1 \neq f_2$ are both contained in a genus $g$ of discriminant $\Delta$, and $G$ is a genus of discriminant $\Delta p^2$. Combining \eqref{pro1}, \eqref{pro2}, along with the fact that every form in $G$ is contained in $\Psi_{G,p}(f)$ for some $f\in g$, gives that $G$ is partitioned into equal size sets by the $\Psi_{G,p}(f)$ with $f\in g$. Hence $|G|=|\Psi_{G,p}(f)|\cdot |g|$. If we sum both sides of the identity in Theorem \ref{newthm}, over $f \in g$, then we arrive at Corollary \ref{genflaw}. The details will be provided in a in a subsequent paper, where we will also deliver the proof of Theorem \ref{newthm}.


Section \ref{example} gives multiple examples of deriving Lambert series decompositions for fundamental discriminants. We now give an example which illustrates how to use repeated applications of Corollary \ref{genflaw} to derive a Lambert series decomposition for the theta series associated with a genus of nonfundamental discriminant. Also this example shows that it is possible to have a Lambert series decomposition associated to a single form when the form is not alone in its genus.\\

We consider the nonfundamental discriminant $\Delta=-63$ and $p=2$ in Corollary \ref{genflaw}. The genus structure for discriminants $-63$ and $-63\cdot 4=-252$ are given below.
\begin{center}
\begin{tabular}{ | l | l | l | l | }
  \hline     
  \multicolumn{2}{|c|}{CL$(-63) \cong \mathbb{Z}_4$}& $\left(\tf{r}{3}\right)$ & $\left(\tf{r}{7}\right)$\\
  \hline                   
   $g_1$ & $(1,1,16)$, $(4,1,4)$ &$+1$ & $+1$ \\ \hline 
	$g_2$ & $(2,1,8)$, $(2,-1,8)$ &$-1$ & $+1$ \\ \hline 
\end{tabular}
\begin{flushright}
			 .
			 \end{flushright}
   \end{center}

\begin{center}
\begin{tabular}{ | l | l | l | l | }
  \hline     
  \multicolumn{2}{|c|}{CL$(-252) \cong \mathbb{Z}_4$}& $\left(\tf{r}{3}\right)$ & $\left(\tf{r}{7}\right)$\\
  \hline                   
   $G_1$ & $(1,0,63)$, $(7,0,9)$ &$+1$ & $+1$ \\ \hline 
	$G_2$ & $(8,6,9)$, $(8,-6,9)$ &$-1$ & $+1$ \\ \hline 
\end{tabular}
\begin{flushright}
			 .
			 \end{flushright}
   \end{center}

Corollary \ref{genflaw} yields
\begin{align}
(1,0,63,q)+(7,0,9,q)&=(1,1,16,q^{4})+(4,1,4,q^{4})+P_{2,1}[(1,1,16,q)+(4,1,4,q)],\label{gpf1}\\
(8,6,9,q)&=(2,1,8,q^{4})+P_{2,1}(2,1,8,q).\label{gpf2}
\end{align}

The goal of this example is to find the Lambert series decompositions for $(1,0,63,q)+(7,0,9,q)$ and $(8,6,9,q)$. We need to employ Corollary \ref{genflaw} once more to find the Lambert series decomposition of $(1,1,16,q)+(4,1,4,q)$ and $(2,1,8,q)$. Taking discriminant $\Delta=-7$ and $p=3$ in Corollary \ref{genflaw} gives
\begin{align}
(1,1,16,q)+(4,1,4,q)&=2(1,1,2,q^{9})+P_{3,1}(1,1,2,q),\label{nug}\\
2(2,1,8,q)&=2(1,1,2,q^{9})+P_{3,2}(1,1,2,q).\label{nugg}
\end{align}
Using the well known
\begin{equation}
\label{36pp2}
(1,1,2,q)=1+2\sum_{n=1}^{\infty}\left(\frac{-7}{n}\right) \frac{q^{n}}{1-q^{n}},
\end{equation}
it is not hard to show
\begin{equation}
\label{36ppp}
P_{3,1}(1,1,2,q)-P_{3,2}(1,1,2,q)=2\sum_{n=1}^{\infty}\left(\frac{21}{n}\right) \frac{q^{n}(1-q^{n})}{1-q^{3n}}.
\end{equation}
We define the Lambert series
\[
f(q):=\sum_{n=1}^{\infty}\left(\frac{-7}{n}\right) \frac{q^{n}}{1-q^{n}},
\]
and
\[
g(q):=\sum_{n=1}^{\infty}\left(\frac{21}{n}\right) \frac{q^{n}(1-q^{n})}{1-q^{3n}}.
\]
Adding and subtracting \eqref{nug} and \eqref{nugg} yields the Lambert series decompositions
\begin{align}
(1,1,16,q)+(4,1,4,q)&=2+f(q)+3f(q^9)+g(q),\label{knug}\\
2(2,1,8,q)&=2+f(q)+3f(q^9)-g(q).\label{knugg}
\end{align}
Since we now have the Lambert series decompositions for the genera of discriminant $-63$, we are ready to repeat the process to derive the Lambert series decompositions for the genera of discriminant $-63\cdot 4=-252$. It is easy to show
\begin{equation}
\label{252n}
P_{2,1}~f(q)=f(q)-2f(q^2)+f(q^4),
\end{equation}
and
\begin{equation}
\label{252nn}
P_{2,1}~g(q)=g(q)+2g(q^2)+g(q^4).
\end{equation}
We now have all the tools needed to find the Lambert series decompositions for $(1,0,63,q)+(7,0,9,q)$ and $(8,6,9,q)$.\\
Employing \eqref{gpf1}, \eqref{knug}, \eqref{252n}, and \eqref{252nn}, yields the Lambert series decomposition
\begin{align*}
(1,0,63,q)+(7,0,9,q)&=2+f(q)-2f(q^2)+2f(q^4)\\
&+3f(q^9)-6f(q^{18})+6f(q^{36})\\
&+g(q)+2g(q^2)+2g(q^4).
\end{align*}
Similarly, we combine \eqref{gpf2}, \eqref{knugg}, \eqref{252n}, and \eqref{252nn}, to find the Lambert series decomposition
\begin{align*}
2(8,6,9,q)&=2+f(q)-2f(q^2)+2f(q^4)\\
&+3f(q^9)-6f(q^{18})+6f(q^{36})\\
&-g(q)-2g(q^2)-2g(q^4).
\end{align*}

We note that one could easily derive product representation formulas for $(1,0,63;n)+(7,0,9;n)$ and $(8,6,9;n)$. In the preceding example, we derived the Lambert series decomposition for the genera of discriminant $-252=-7\cdot 2^2 \cdot 3^2$ by using Corollary \ref{genflaw} on the pairs $(\Delta,p)$ equals $(-63,2)$ and $(-7,3)$. We remark that we could have taken an alternate route and used Corollary \ref{genflaw} on the pairs $(\Delta,p)$ equals $(-28,3)$ and $(-7,2)$ to find a Lambert series decomposition for the genera of discriminant $-252$. \\

We now give a new eta-quotient identity for discriminant $-252$.
\begin{theorem}
\label{et}
Using the standard notation of this paper we have
\begin{equation}
\label{focus}
\frac{(1,0,63,q) - (7,0,9,q)}{2} = qE(-q^3)E(-q^{21}).
\end{equation}
\end{theorem}

\begin{proof}
 We begin by employing Theorem 2.1 of \cite[pg.191]{berpat} with $s=1, m=3$ to find
\begin{equation}
\label{6m}
\frac{(1,1,16,q) - (4,1,4,q)}{2} = qE(q^3)E(q^{21}).
\end{equation}
Writing $x\equiv y$ to mean $x\equiv y \imod{2}$ we have
\begin{equation}
\label{n1}
\begin{aligned}
P_{2,1}(1,1,16,q)&=\sum_{\substack{x\not\equiv 0\\  y\equiv 0}}q^{x^{2}+xy+16y^{2}}\\
&=\sum_{\substack{x \\ y\equiv 0}}q^{x^{2}+xy+16y^{2}}-\sum_{\substack{x\equiv 0\\  y\equiv 0}}q^{x^{2}+xy+16y^{2}}\\
&=\sum_{x,y}q^{x^2+2xy+64y^2}-(1,1,16,q^{4})\\
&=\sum_{x,y}q^{(-x-y)^2+2(-x-y)y+64y^2}-(1,1,16,q^{4})\\
&=(1,0,63,q)-(1,1,16,q^{4}),
\end{aligned}
\end{equation}
and
\begin{equation}
\label{n2}
\begin{aligned}
P_{2,1}(4,1,4,q)&=\sum_{\substack{x\not\equiv 0\\  y\not\equiv 0}}q^{4x^{2}+xy+4y^{2}}\\
&=\sum_{x\equiv y}q^{4x^{2}+xy+4y^{2}}-\sum_{\substack{x\equiv 0\\  y\equiv 0}}q^{4x^{2}+xy+4y^{2}}\\
&=\sum_{x,y}q^{4(y-x)^{2}+(y-x)(y+x)+4(y+x)^{2}}-(4,1,4,q^{4})\\
&=(7,0,9,q)-(4,1,4,q^{4}).
\end{aligned}
\end{equation}
Subtracting \eqref{n1} and \eqref{n2} gives
\begin{equation}
\label{ll}
\frac{(1,0,63,q) - (7,0,9,q)}{2}=\frac{(1,1,16,q^4) - (4,1,4,q^4)}{2} + P_{2,1}\frac{(1,1,16,q) - (4,1,4,q)}{2}.
\end{equation}
It is clear from \eqref{ll} that
\begin{equation}
\label{o}
P_{2,1}\frac{(1,1,16,q) - (4,1,4,q)}{2} = P_{2,1}\frac{(1,0,63,q) - (7,0,9,q)}{2}.
\end{equation}
One can directly show
\begin{equation}
\label{even4}
P_{2,0}(4,1,4,q)= 2(2,1,8,q^2)-(4,1,4,q^4),
\end{equation}
and
\begin{equation}
\label{even5}
P_{2,0}(1,1,16,q)= 2(2,1,8,q^2)-(1,1,16,q^4).
\end{equation}
Combining \eqref{ll}, \eqref{even4}, and \eqref{even5}, yields
\begin{equation}
\label{e}
P_{2,0}\frac{(1,1,16,q) - (4,1,4,q)}{2} = -P_{2,0}\frac{(1,0,63,q) - (7,0,9,q)}{2}.
\end{equation}
Employing \eqref{6m}, \eqref{o}, and \eqref{e} gives \eqref{focus}. We note that \eqref{focus} is not covered in \cite[Thm. 2.1]{berpat} and is new.

\end{proof}

      \section{Acknowledgements}
   \label{ack}
   
	I am grateful to Duncan Buell for helpful comments and suggestions, and also to Alexander Berkovich for his advice and guidance.

\end{document}